\documentclass[12pt]{amsart}

\def\bR {\mathbf{R}}

\def\cD {\mathcal{D}}
\def\cE {\mathcal{E}}

\def\cP {\mathcal{P}}

\def\cW {\mathcal{W}}

\def\de {{\delta}}
\def\eps {{\epsilon}}
\def\th {{\theta}}

\def\d {{\partial}}
\def\grad {{\nabla}}
\def\Dlt {{\Delta}}

\def\rstr {{\big |}}
\def\indc {{\bf 1}}

\def\Lip{{\mathrm{Lip}}}

\newcommand{\Div}{\operatorname{div}}
\newcommand{\Rot}{\operatorname{curl}}

\newcommand{\Supp}{\operatorname{supp}}

\newcommand{\Dist}{\operatorname{dist}}

\newcommand{\ba}{\begin{aligned}}
\newcommand{\ea}{\end{aligned}}

\newcommand{\be}{\begin{equation}}
\newcommand{\ee}{\end{equation}}

\newcommand{\lb}{\label}

\newtheorem{Thm}{Theorem}[section]

\newtheorem{Prop}[Thm]{Proposition}
\newtheorem{Cor}[Thm]{Corollary}

\newtheorem{Def}[Thm]{Definition}

\begin{document}

\title[Mean-Field Limit for Vlasov-Maxwell]{The Mean-Field Limit for a Regularized Vlasov-Maxwell Dynamics}

\author[F. Golse]{Fran\c cois Golse}

\address[F. G.]%
{Ecole polytechnique\\
Centre de math\'ematiques L. Schwartz\\
F91128 Palaiseau cedex\\
\& Universit\'e Paris-Diderot\\
Laboratoire J.-L. Lions, BP 187\\
F75252 Paris cedex 05} 

\email{francois.golse@math.polytechnique.fr}

\keywords{Mean-field limit; Vlasov-Maxwell system; Li\'enard-Wiechert potential, Monge-Kantorovich-Rubinstein distance} 

\subjclass{82C22, 35Q83, 35Q61 (82D10)}

\begin{abstract}
The present work establishes the mean-field limit of a $N$-particle system towards a regularized variant of the relativistic Vlasov-Maxwell system,
following the work of Braun-Hepp [Comm. in Math. Phys. \textbf{56} (1977), 101--113] and Dobrushin [Func. Anal. Appl. \textbf{13} (1979), 115--123]
for the Vlasov-Poisson system. The main ingredients in the analysis of this system are (a) a kinetic formulation of the Maxwell equations in terms of 
a distribution of electromagnetic potential in the momentum variable, (b) a regularization procedure for which an analogue of the total energy --- i.e.
the kinetic energy of the particles plus the energy of the electromagnetic field --- is conserved and (c) an analogue of Dobrushin's stability estimate
for the Monge-Kantorovich-Rubinstein distance between two solutions of the regularized Vlasov-Poisson dynamics adapted to retarded potentials.
\end{abstract}

\maketitle

\section{Introduction}

Vlasov equations are kinetic equations used to describe the evolution of dilute, collisionless systems of particles coupled by some long-range
interaction. For instance, in the case of a rarefied plasma, each charged particle is accelerated by the electrostatic field created by all the other 
particles, a situation described by the Vlasov-Poisson system \cite{Vlasov38}. A similar Vlasov-Poisson system is used in astrophysics to model 
the collective behavior of like massive objects coupled by Newton's gravitational force --- which is attractive, at variance with the electrostatic
force between charged particles with charges of the same sign.

Other types of interactions can also be considered. In the case of magnetized plasmas --- e.g. plasmas in tokamak machines --- each charged
particle is accelerated by the Lorentz force resulting from the electromagnetic field created by all the other particles, a situation described by
the Vlasov-Maxwell system \cite{LuVla50}. There are also general relativistic variants of the gravitational Vlasov-Poisson system, such as the 
Vlasov-Einstein \cite{ChoquetBruhat71} or the (model) Vlasov-Nordstr\"om systems \cite{CaloRein03}.

At the time of this writing, there is no rigorous derivation of any of these kinetic equations from the corresponding $N$-body problems, in the
large $N$ limit. Braun and Hepp \cite{BraunHepp} and Dobrushin \cite{Dobrushin} have proposed instead rigorous derivations of a system
analogous to the Vlasov-Poisson system, replacing the Coulomb potential with a twice differentiable mollification thereof. Hauray and Jabin
\cite{HaurayJabin07} have succeeded in treating the case of singular potentials, but their analysis does not include the Coulomb singularity 
yet. It should be noted that the situation is significantly better in the quantum case: the Hartree or Schr\"odinger-Poisson equation,  that is the 
quantum analogue of the Vlasov-Poisson system has been derived from the quantum $N$-body problem with Coulomb potential by Erd\"os 
and Yau \cite{ErdosYau}.

In the case of the Vlasov-Maxwell system, it is not even clear what the corresponding $N$-body problem should be. For instance, there 
is at present no satisfying description of the electromagnetic self-interaction (i.e. the action of the electromagnetic field created by a moving, 
point-like charge on itself) within the theory of classical electrodynamics. Bibliographic references discussing this well-known issue can be
found in the remarks following the introduction of the regularized $N$-body system (\ref{DynSystVMReg}).

The present paper proposes an analogue of the Braun-Hepp or Dobrushin theory for (a mollified version of) the Vlasov-Maxwell system. 
Elskens, Kiessling and Ricci \cite{ElsKiesRic09} have recently considered a special relativistic variant of the gravitational Vlasov-Poisson 
system where the Poisson equation for the potential is replaced with a linear wave equation. 

The problem of deriving a regularized variant of the Vlasov-Maxwell equations from a particle system is explicitly mentioned by Kiessling on 
p. 111 in his survey article \cite{Kiessling08}. As expected, this problem can be handled more or less along the line of \cite{ElsKiesRic09}, with 
however a few significant modifications, which it is the purpose of the present work to explain. As we shall see, the regularization procedure
removes all the conceptual difficulties pertaining to the electromagnetic self-interaction mentioned above. The idea of considering a mollified
variant of the Vlasov-Maxwell dynamics is therefore motivated by reasons other than analytical simplicity.

\section{A scalar formulation of the Vlasov-Maxwell system}

Regularizing the Coulomb potential is a crucial step in the mean-field limit established by Braun and Hepp \cite{BraunHepp} and Dobrushin 
\cite{Dobrushin}. At first sight, it is not clear that the same procedure can be applied to the Vlasov-Maxwell system, as the electromagnetic
field involves both a scalar and vector potentials. Perhaps for that reason, Elskens, Kiessling and Ricci \cite{ElsKiesRic09} chose ``to purge
the inhomogeneous wave equation for [the vector potential] $A$ and all terms involving (derivatives of) $A$ in the Lorentz force'' and retain
only the inhomogeneous wave equation for the scalar potential.

In this section, we explain how both the scalar and vector potentials in the Vlasov-Maxwell system can be expressed in terms of a single scalar
potential distributed in the momentum variable. This formulation, introduced in \cite{BGPArch,BGPIbero} for other purposes --- viz. for studying
regularity properties of the Vlasov-Maxwell system --- although perhaps not absolutely necessary in the present context, greatly simplifies our 
arguments in the sequel.

The Vlasov-Maxwell system for a single species of particles with momentum distribution function $f(t,x,\xi)$ --- meaning that $f(t,x,\xi)$ is the 
density  of particles with momentum $\xi\in\bR^3$ that are located at the position $x\in\bR^3$ at time $t\ge 0$ --- is written as follows:
\be\lb{VM}
\left\{
\begin{array}{l}
\d_tf+v(\xi)\cdot\grad_xf+(E+v(\xi)\wedge B)\cdot\grad_\xi f=0\,,
\\
\Div_xB=0\,,
\\
\d_tB+\Rot_xE=0\,,
\\
\Div_xE=\rho_f\,,
\\
\d_tE-\Rot_xB=-j_f\,.
\end{array}
\right.
\ee
We have used the notation
$$
v(\xi):=\grad_\xi e(\xi)=\frac{\xi}{\sqrt{1+|\xi|^2}}\,,\quad\hbox{ where }e(\xi):=\sqrt{1+|\xi|^2}
$$
and
$$
\rho_f(t,x):=\int_{\bR^3}f(t,x,\xi)d\xi\,,\quad\hbox{ while }j_f(t,x):=\int_{\bR^3}v(\xi)f(t,x,\xi)d\xi\,.
$$

This system is posed in the whole Euclidian space --- i.e. for all $x,\xi\in\bR^3$ --- and supplemented with the initial data
\be\lb{CondIn}
f\rstr_{t=0}=f^{in}\,,\quad E\rstr_{t=0}=E^{in}\,,\quad B\rstr_{t=0}=0\,,
\ee
where $E^{in}$ is the electrostatic field created by the distribution of charges $f^{in}$:
\be\lb{CompatIn}
E^{in}=-\grad_x\phi^{in}\,,\quad\phi^{in}=(-\Dlt_x)^{-1}\int_{\bR^3}f^{in}d\xi\,.
\ee
(We restrict our attention to such initial data for simplicity; treating the case of general initial data 
$$
f\rstr_{t=0}=f^{in}\,,\quad E\rstr_{t=0}=E^{in}\,,\quad B\rstr_{t=0}=B^{in}\,,
$$
with the compatibility conditions
$$
\Div_x E^{in}=\int_{\bR^3}f^{in}d\xi\,,\quad\Div_x B^{in}=0
$$
is not significantly more complicated, as it only involves one extra linear, homogeneous wave equation.) 

Solve for $\phi_0$ the Cauchy problem for the linear wave equation
\be\lb{DefPhi0}
\left\{
\begin{array}{l}
\Box_{t,x}\phi_0=0\,,
\\	\\
\phi_0\rstr_{t=0}=\phi^{in}=(-\Dlt_x)^{-1}\displaystyle\int_{\bR^3}f^{in}d\xi\,,
\\	\\
\d_t\phi_0\rstr_{t=0}=0\,.
\end{array}
\right.
\ee

The scalar formulation of (\ref{VM}) is as follows. Consider the coupled system with unknowns $f\equiv f(t,x,\xi)\ge 0$ and $u\equiv u(t,x,\xi)\in\bR$:
\be\lb{VMScal}
\left\{
\begin{array}{l}
\d_tf+v(\xi)\cdot\grad_xf+(E+v(\xi)\wedge B)\cdot\grad_\xi f=0\,,
\\	\\
\Box_{t,x} u=f\,,
\\	\\
\phi=\phi_0+\displaystyle\int_{\bR^3}ud\xi\,,\quad A=\displaystyle\int_{\bR^3}v(\xi)ud\xi\,,
\\	\\
B=\Rot_xA\,,\quad E=-\d_tA-\grad_x\phi\,,
\end{array}
\right.
\ee
with initial data
\be\lb{CondInVMScal}
f\rstr_{t=0}=f^{in}\,,\quad u\rstr_{t=0}=\d_tu\rstr_{t=0}=0\,.
\ee
Physically, $u(t,x,\xi)$ is the instantaneous Li\'enard-Wiechert potential created by a particle with momentum $\xi$ at the position $x$, distributed 
under $f(t,x,\xi)$.

\begin{Prop}\lb{P-ScalReduc}
Let $f\in C_c^\infty([0,T]\times\bR^3\times\bR^3)$, and let $\phi_0$ be the solution of (\ref{DefPhi0}). If $(f,u)$ satisfies (\ref{VMScal}) with initial 
data (\ref{CondInVMScal}), then the electromagnetic potential $(\phi,A)$ defined as in (\ref{VMScal}) is smooth and satisfies the Lorentz gauge
$$
\d_t\phi+\Div_xA=0\,,
$$
while $(f,E,B)$, with the electromagnetic field $(E,B)$ as in (\ref{VMScal}), satisfies (\ref{VM}) with initial condition (\ref{CondIn})-(\ref{CompatIn}).
\end{Prop}

\smallskip
We shall use systematically the following elements of notation: $C_c(X)$ (resp. $C^k_c(X)$, $C^\infty_c(X)$) designates the set of continuous
(resp. $C^k$, $C^\infty$) functions with compact support included in $X$. Likewise, $C_b(X)$ designates the set of bounded continuous functions 
defined on $X$, while $C^k_b(X)$ designates the set of $C^k$ functions defined on $X$ whose derivatives of order $\le k$ are bouned on $X$. 
Finally, $C^\infty_b(X)$ designates the class of $C^\infty$ defined on $X$ whose derivatives of all orders are bounded on $X$.

\begin{proof}
Since $f^{in}\in C_c^\infty(\bR^3\times\bR^3) $, the initial electrostatic potential 
$$
\phi^{in}=(-\Dlt_x)^{-1}\rho_{f^{in}}\in C_b^\infty(\bR^3)\,,
$$
so that the solution $\phi_0$ of the Cauchy problem (\ref{DefPhi0}) is in $C_b^\infty([0,T]\times\bR^3)$ for each $T>0$. 

Since $f\in C_c^\infty([0,T]\times\bR^3\times\bR^3)$, the finite speed of propagation and the regularity theory for solutions of the wave equation
implies that $u\in C_c^\infty([0,T]\times\bR^3\times\bR^3)$. 

Therefore the scalar potential $\phi\in C_b^\infty([0,T]\times\bR^3)$ and the vector potential $A\in C_c^\infty([0,T]\times\bR^3;\bR^3)$, so that, in 
view of the last equalities in (\ref{VMScal}), the electric field $E\in C^\infty_b([0,T]\times\bR^3;\bR^3)$ while the magnetic field 
$B\in C_c^\infty([0,T]\times\bR^3;\bR^3)$ and
$$
\left\{
\ba
{}&\Div_xB=0\,,
\\
&\d_tB+\Rot_xE=0\,.
\ea
\right.
$$

Besides, $\rho_f$ and $j_f$ belong to $C_c^\infty([0,T]\times\bR^3)$ and $C_c^\infty([0,T]\times\bR^3;\bR^3)$ respectively, and since
$v(\xi)=\grad_\xi e(\xi)$, one has $\Div_\xi(v(\xi)\wedge B(t,x))=B(t,x)\cdot\Rot_\xi v(\xi)=0$ so that
$$
\ba
\d_t\rho_f+\Div_xj_f&=\int_{\bR^3}(\d_tf+v(\xi)\cdot\grad_xf)(t,x,\xi)d\xi
\\
&=\int_{\bR^3}\Div_\xi(-(E+v\wedge B)f)(t,x,\xi)d\xi=0\,.
\ea
$$
Let us verify that $(\phi,A)$ satisfies the Lorentz gauge condition. Setting $g:=\d_t\phi+\Div_xA\in C_b^\infty([0,T]\times\bR^3)$, one has
$$
\ba
\Box_{t,x} g&=\Box_{t,x}\d_t\phi_0+\Box_{t,x}\int_{\bR^3}(\d_tu+\Div_x(v(\xi)u))d\xi
\\
&=\d_t\rho_f+\Div_x j_f=0\,,
\ea
$$
together with the initial conditions
$$
g\rstr_{t=0}=\d_t\phi_0\rstr_{t=0}+\int_{\bR^3}(\d_tu\rstr_{t=0}+\Div_x(v(\xi)u\rstr_{t=0}))d\xi=0\,,
$$
and
$$
\ba
\d_tg\rstr_{t=0}&=\d_t^2\phi\rstr_{t=0}+\Div_x\d_tA\rstr_{t=0}
\\
&=\Box_{t,x}\phi_0\rstr_{t=0}\!+\!\Delta_x\phi\rstr_{t=0}\!+\!\int_{\bR^3}\d_t^2u\rstr_{t=0}d\xi\!+\!\int_{\bR^3}v(\xi)\cdot\grad_x\d_tu\rstr_{t=0}d\xi
\\
&=\Delta_x\phi^{in}+\int_{\bR^3}\Box_{t,x} u\rstr_{t=0}d\xi+\int_{\bR^3}\Dlt_xu\rstr_{t=0}d\xi
\\
&=\Delta_x\phi^{in}+\int_{\bR^3}f^{in}d\xi=0\,.
\ea
$$
By the uniqueness property for the wave equation, one concludes that $g=0$, which is the Lorentz gauge condition.

It remains to verify that the electromagnetic field satisfies the Gauss and Maxwell-Amp\`ere equations: by the Lorentz gauge condition,
$$
\ba
\Div_xE&=\Box_{t,x}\phi-\d_t(\d_t\phi+\Div_xA)
\\
&=\Box_{t,x}\phi_0+\Box_{t,x}\int_{\bR^3}ud\xi=\int_{\bR^3}fd\xi=\rho_f\,,
\ea
$$
while
$$
\ba
\d_tE-\Rot_xB&=-\Box_{t,x} A-\grad_x(\d_t\phi+\Div_xA)
\\
&=-\Box_{t,x}\int_{\bR^3}v(\xi)ud\xi=-\int_{\bR^3}v(\xi)fd\xi=-j_f\,.
\ea
$$
Therefore $(f,E,B)$ is a solution of the Vlasov-Maxwell system (\ref{VM}), and
$$
\ba
E\rstr_{t=0}=-\int_{\bR^3}v(\xi)\d_tu\rstr_{t=0}d\xi-\grad_x\int_{\bR^3}u\rstr_{t=0}d\xi-\grad_x\phi_0\rstr_{t=0}
\\
=-\grad_x(-\Delta_x)^{-1}\int_{\bR^3}f^{in}d\xi=E^{in}\,,
\ea
$$
while
$$
B\rstr_{t=0}=\Rot_x\int_{\bR^3}v(\xi)u\rstr_{t=0}d\xi=0\,,
$$
so that $(E,B)$ also satisfies the initial condition (\ref{CondIn}).
\end{proof}

With this formulation of the Vlasov-Maxwell system, the analogue of the Coulomb potential in the Vlasov-Poisson system becomes fairly obvious.
Let $Y$ be the forward fundamental solution of the d'Alembert operator, i.e. the only $Y\in\cD'(\bR\times\bR^3)$ satisfying
\be\lb{DefY}
\left\{
\begin{array}{l}
\Box_{t,x} Y=\de_{(t,x)=(0,0)}\,,
\\
\Supp(Y)\subset\bR_+\times\bR^3\,.
\end{array}
\right.
\ee
We recall that, in space dimension $3$, the distribution $Y$ is given by the formula
\be\lb{FlaY}
Y=\frac{\indc_{t>0}}{4\pi t}\de(|x|-t)\,,
\ee
where the notation $\de(|x|-t)$ is understood as the surface measure on the sphere of radius $t>0$ centered at the origin. Then, the function $u$ 
in (\ref{VMScal}) is given in terms of $f$ by the formula
\be\lb{f->u}
u(\cdot,\cdot,\xi)=Y\star_{t,x}(\indc_{t>0}f(\cdot,\cdot,\xi))\,,\quad\xi\in\bR^3
\ee
where $\star_{t,x}$ denotes convolution in the variables $(t,x)$, while the scalar potential $\phi_0$ is given in terms of $f^{in}$ by
\be\lb{f->phi0}
\phi_0(t,\cdot)=\d_tY(t,\cdot)\star_{x}G\star_x\int_{\bR^3}f^{in}(\cdot,\xi)d\xi\,,\quad t>0\,,
\ee
where
\be\lb{DefG}
G(z)=\frac{1}{4\pi|z|}\,,\quad z\in\bR^3
\ee
is the fundamental solution of $-\Dlt_x$.

Therefore, the Lorentz force field $F:=E+v\wedge B$ in the Vlasov-Maxwell system is given in terms of the particle distribution function $f$ 
by the formula
$$
\ba
F[f]=&-\int_{\bR^3}\grad_x\d_tY(t,\cdot)\star_{x}G\star_xf(0,\cdot,\eta)d\eta
\\
&-\int_{\bR^3}(\grad_x+v(\eta)\d_t)Y\star_{t,x}f(\cdot,\cdot,\eta)d\eta
\\
&-\int_{\bR^3}v(\xi)\wedge(v(\eta)\wedge\grad_xY\star_{t,x}f(\cdot,\cdot,\eta))d\eta\,,
\ea
$$
to be compared with the formula
$$
F_{\mathrm{Poisson}}[f]=-\int_{\bR^3}\grad_xG\star_xf(t,\cdot,\eta)d\eta
$$
in the case of the Vlasov-Poisson system. 

In the Vlasov-Poisson case, the $N$-particle dynamics considered by Hauray and Jabin (see for instance formula (1.1)-(1.2) in \cite{HaurayJabin07})
is the differential system
\be\lb{DynSystVP}
\ddot{x}_i(t)=-\frac1{N-1}\sum_{j=1\atop j\not=i}^N\grad_xG(x_i(t)-x_j(t))\,,\quad 1\le i\le N\,.
\ee

By analogy, in the Vlasov-Maxwell case, one should consider the system
\be\lb{DynSystVM}
\left\{
\ba
{}&\dot{x}_i(t)=v(\xi_i(t))
\\
&\dot\xi_i(t)=-\frac1{N-1}\sum_{j=1\atop j\not=i}^N\grad_x\d_tY(t,\cdot)\star_xG(x_i(0)-x_j(0))
\\
&-\frac1{N-1}\sum_{j=1\atop j\not=i}^N\int_0^t(\grad_x+v(\xi_j(s))\d_t)Y(t-s,x_i(t)-x_j(s))ds
\\
&-\frac1{N-1}\sum_{j=1\atop j\not=i}^N\int_0^tv(\xi_i(t))\wedge(v(\xi_j(s))\wedge\grad_xY(t-s,x_i(t)-x_j(s)))ds\,,
\\
&\hskip9.8cm {i=1,\ldots,N\,.}
\ea
\right.
\ee

There is an obvious difficulty in (\ref{DynSystVP}): since $\grad_xG$ is singular at the origin, the right-hand side of the equation for $\dot{x}_i$
is not defined whenever $x_i(t)=x_j(t)$ for some $j\not=i$. More seriously, in (\ref{DynSystVM}), the second and third terms in the r.h.s. of the 
equation for $\xi_i$ are not a priori well defined quantities because $Y$ is a measure and not a ($C^1$) function. Yet this difficulty disappears 
in the reference \cite{BraunHepp}, where the Green function $G$ is replaced with a regularization thereof. Likewise, we shall replace $Y$ 
with some regularized variant thereof, which removes the difficulty mentioned above.

Notice that, while (\ref{DynSystVP}) is a system of $N$ coupled ordinary differential equations, (\ref{DynSystVM}) is a system of $N$ coupled
integro-differential equations --- more specifically, delay differential equations. The same is true of the dynamical system (7)-(9) considered
in \cite{ElsKiesRic09}, and this is a consequence of the retarded potential formula for the solution of the Maxwell system, or, equivalently, of
Kirchhoff's formula (\ref{f->u}) for the solution of the wave equation.

\section{A regularization of the Vlasov-Maxwell system}

Any regularization of $Y$ would transform (\ref{DynSystVM}) into a well-posed system. Yet, classical solutions of the original Vlasov-Maxwell 
system satisfy some conservation laws, which one might wish to preserve by the regularization procedure.

Since $v=\grad_\xi e$, 
\be\lb{DivFreeLorentz}
\Div_{x,\xi}(v(\xi),E+v\wedge B)=\Div_\xi(v\wedge B)=B\cdot\Rot_\xi v=0
\ee
so that the characteristic field of the transport equation in (\ref{VM}) preserves the phase space Lebesgue measure $dxd\xi$. Hence, if $(f,E,B)$ 
is a classical solution of (\ref{VM}) on the time interval $[0,T]$, then, for all $p\in[1,+\infty]$, one has
\be\lb{ConsLp}
\|f(t,\cdot,\cdot)\|_{L^p(\bR^3\times\bR^3)}=\|f^{in}(\cdot,\cdot)\|_{L^p(\bR^3\times\bR^3)}
\ee
for each $t\in[0,T]$. This obviously remains true if the electromagnetic field $(E,B)$ is replaced in the transport equation governing $f$ with any 
regularization thereof in the $t,x$ variables.

Moreover, if $(f,E,B)$ is a classical solution of (\ref{VM}) on the time interval $[0,T]$ with, for instance, 
$f\in C^1_c([0,T]\times\bR^3\times\bR^3)$ while the electromagnetic field $E,B\in C^1_b([0,T]\times\bR^3;\bR^3)\cap C([0,T];L^2(\bR^3;\bR^3))$, 
then
\be\lb{ConsEnerg}
\iint_{\bR^3\times\bR^3}e(\xi)f(t,x,\xi)dxd\xi+\tfrac12\int_{\bR^3}(|E|^2+|B|^2)(t,x)dx=\hbox{Const.,}
\ee
which is the conservation of energy (the first integral being the kinetic energy of the particles while the second is the electromagnetic energy).
Conservation of energy is not expected to be preserved by all regularizations of the electromagnetic field. 

There is however one clever regularization procedure for the Vlasov-Maxwell system that preserves a quantity analogous to the energy and
which converges to the energy (\ref{ConsEnerg}) as the regularization is removed. This regularization of the Vlasov-Maxwell system, proposed
by Rein in \cite{ReinCMS}, is recalled below for the reader's convenience\footnote{Prof. Rein kindly informed the author that this regularization
procedure had been introduced earlier in this context by E. Horst in his Habilitationsschrift.}.

Let $\chi\in C^\infty_c(\bR^3)$ satisfy
$$
\chi(x)=\chi(-x)\ge 0\,,\quad\Supp(\chi)\subset B(0,1)\,,\quad\int_{\bR^3}\chi(x)dx=1\,,
$$
and define the regularizing sequence
$$
\chi_\eps(x)=\frac1{\eps^3}\chi\left(\frac{x}\eps\right)
$$
for each $\eps>0$. Consider then, for each $\eps>0$, the following regularized variant of (\ref{VMScal}) with unknown $(f_\eps,u_\eps)$:
\be\lb{VMScalReg}
\left\{
\begin{array}{l}
\d_tf_\eps+v(\xi)\cdot\grad_xf_\eps+(E_\eps+v(\xi)\wedge B_\eps)\cdot\grad_\xi f_\eps=0\,,
\\	\\
\Box_{t,x} u_\eps=\chi_\eps\star_x\chi_\eps\star_xf_\eps\,,
\\	\\
\phi_\eps=\chi_\eps\star_x\chi_\eps\star_x\phi_0+\displaystyle\int_{\bR^3}u_\eps d\xi\,,\quad A_\eps=\displaystyle\int_{\bR^3}v(\xi)u_\eps d\xi\,,
\\	\\
B_\eps=\Rot_xA_\eps\,,\quad E_\eps=-\d_tA_\eps-\grad_x\phi_\eps\,,
\end{array}
\right.
\ee
where $\phi_0$ is defined in as (\ref{DefPhi0}), and with the same initial data as for (\ref{VMScal}):
\be\lb{CondInVMScalReg}
f_\eps\rstr_{t=0}=f^{in}\,,\quad u_\eps\rstr_{t=0}=\d_tu_\eps\rstr_{t=0}=0\,.
\ee

Along with the electromagnetic field $(E_\eps,B_\eps)$ of the regularized system (\ref{VMScalReg}), consider $(\tilde E_\eps,\tilde B_\eps)$ defined
as the solution of the Maxwell system
\be\lb{Maxw}
\left\{
\begin{array}{l}
\Div_x\tilde B_\eps=0\,,
\\
\d_t\tilde B_\eps+\Rot_x\tilde E_\eps=0\,,
\\
\Div_x\tilde E_\eps=\chi_\eps\star_x\rho_{f_\eps}\,,
\\
\d_t\tilde E_\eps-\Rot_x\tilde B_\eps=-\chi_\eps\star_xj_{f_\eps}\,,
\end{array}
\right.
\ee
with initial data
\be\lb{CondinMaxw}
\tilde E_\eps\rstr_{t=0}=\chi_\eps\star_xE^{in}\,,\quad\tilde B_\eps\rstr_{t=0}=0\,.
\ee

\begin{Prop}[Rein \cite{ReinCMS}]\lb{P-ExistUniqVMeps}
Let $\eps>0$ and $f^{in}\in L^\infty(\bR^3\times\bR^3)$ be such that 
$$
f^{in}\ge 0\hbox{ a.e. on }\bR^3\times\bR^3\hbox{ and }\Supp(f^{in})\hbox{ is compact.}
$$ 
Then the regularized system (\ref{VMScalReg}) has a unique weak solution $(f_\eps,u_\eps)$ defined for all positive times and satisfying 
$u_\eps\in C_c^\infty([0,T]\times\bR^3\times\bR^3)$ for each $T>0$, while 
$$
\ba
{}&f_\eps(t,x,\xi)\ge 0\quad\hbox{ for a.e. }(t,x,\xi)\in\bR_+\times\bR^3\times\bR^3\,,
\\
&\|f_\eps(t,\cdot,\cdot)\|_{L^p(\bR^3\times\bR^3)}=\hbox{Const.}\hbox{ for all $p\in[1,+\infty]$, and}
\\
&\iint_{\bR^3\times\bR^3}e(\xi)f_\eps(t,x,\xi)dxd\xi+\tfrac12\int_{\bR^3}(|\tilde E_\eps|^2+|\tilde B_\eps|^2)(t,x)dx=\hbox{Const.}
\ea
$$
for each $t\ge 0$.
\end{Prop}

\begin{proof}
The same computation as in the proof of Proposition \ref{P-ScalReduc} shows that $(f_\eps,E_\eps,B_\eps)$ satisfies
\be\lb{VMReg}
\left\{
\begin{array}{l}
\d_tf_\eps+v(\xi)\cdot\grad_xf_\eps+(E_\eps+v(\xi)\wedge B_\eps)\cdot\grad_\xi f_\eps=0\,,
\\
\Div_xB_\eps=0\,,
\\
\d_tB_\eps+\Rot_xE_\eps=0\,,
\\
\Div_xE_\eps=\chi_\eps\star_x\chi_\eps\star_x\rho_{f_\eps}\,,
\\
\d_tE_\eps-\Rot_xB_\eps=-\chi_\eps\star_x\chi_\eps\star_xj_{f_\eps}\,,
\end{array}
\right.
\ee
with initial data
\be\lb{CondInReg}
f_\eps\rstr_{t=0}=f^{in}\,,\quad E_\eps\rstr_{t=0}=\chi_\eps\star_x\chi_\eps\star_xE^{in}\,,\quad B_\eps\rstr_{t=0}=0\,,
\ee
where $E^{in}$ is defined as in (\ref{CompatIn}).

This is precisely the regularized Vlasov-Maxwell system studied by Rein in \cite{ReinCMS}, to which we refer for the existence of a solution.

That $f_\eps\ge 0$ a.e. on $\bR_+\times\bR^3\times\bR^3$ and $t\mapsto\|f_\eps(t,\cdot,\cdot)\|_{L^p(\bR^3\times\bR^3})$ is constant on 
$\bR_+$ follows from the method of characteristics for the transport equation satisfied by $f_\eps$ since $E_\eps$ and $B_\eps$ are smooth
with bounded derivatives of all orders on $\bR_+\times\bR^3\times\bR^3$ and the vector field 
$$
(x,\xi)\mapsto(v(\xi),E_\eps(t,x)+v(\xi)\wedge B_\eps(t,x))
$$
is divergence free, while $f^{in}\ge 0$ a.e. on $\bR^3\times\bR^3$ and $f^{in}\in L^p(\bR^3\times\bR^3)$ for each $p\in[1,+\infty]$.

We briefly recall from \cite{ReinCMS} the elegant argument leading to the conservation of the (pseudo-)energy for (\ref{VMReg}), which is the 
reason for the specific regularization procedure chosen in (\ref{VMReg}). For simplicity, we assume that $f^{in}\in C^1_c(\bR^3\times\bR^3)$, 
so that, by the method of characteristics, $f_\eps\in C^1_c([0,T]\times\bR^3\times\bR^3)$ for each $T>0$. Then
\be\lb{Energ1}
\frac{d}{dt}\iint_{\bR^3\times\bR^3}e(\xi)f_\eps(t,x,\xi)dxd\xi=\int_{\bR^3}E_\eps\cdot j_{f_\eps}(t,x)dx\,.
\ee
By uniqueness of the solution of the Maxwell system, $E_\eps=\chi_\eps\star_x\tilde E_\eps$, so that
$$
\ba
\int_{\bR^3}E_\eps\cdot j_{f_\eps}(t,x)dx&=\int_{\bR^3}(\chi_\eps\star_x\tilde E_\eps)\cdot j_{f_\eps}(t,x)dx
\\
&=\int_{\bR^3}\tilde E_\eps\cdot(\chi_\eps\star_xj_{f_\eps})(t,x)dx
\ea
$$
where the last equality follows from the fact that $\chi_\eps$ is even. On the other hand, the classical computation leading to the conservation
of energy in the Maxwell system (\ref{Maxw}) is as follows:
\be\lb{Energ2}
\ba
\frac{d}{dt}\int_{\bR^3}&\tfrac12(|\tilde E_\eps|^2+|\tilde B_\eps|^2)(t,x)dx
\\
&=\int_{\bR^3}(\tilde E_\eps\cdot\d_t\tilde E_\eps+\tilde B_\eps\cdot\d_t\tilde B_\eps)(t,x)dx
\\
&=-\int_{\bR^3}\tilde E_\eps\cdot(\chi_\eps\star_xj_{f_\eps})(t,x)dx
\\
&+\int_{\bR^3}(\tilde E_\eps\cdot\Rot_x\tilde B_\eps-\tilde B_\eps\cdot\Rot_x\tilde E_\eps)(t,x)dx\,.
\ea
\ee
Now the last integral above vanishes since
$$
\tilde E_\eps\cdot\Rot_x\tilde B_\eps-\tilde B_\eps\cdot\Rot_x\tilde E_\eps=-\Div_x(E_\eps\wedge B_\eps)
$$
and $B_\eps(t,\cdot)$ has compact support for each $t\ge 0$. Combining (\ref{Energ1}) and (\ref{Energ2}) give the announced conservation
of (pseudo-)energy for (\ref{VMReg}), and therefore for (\ref{VMScalReg}).
\end{proof}

\smallskip
At this point, we introduce the regularized $N$-particle dynamics for the Vlasov-Maxwell system. Define
\be\lb{DefYeps}
Y_\eps=\chi_\eps\star_x\chi_\eps\star_xY\,.
\ee
In other words, $Y_\eps$ is the solution of 
$$
\left\{
\ba
{}&\Box_{t,x} Y_\eps=\de_{t=0}\otimes(\chi_\eps\star_x\chi_\eps)\,,\quad\hbox{ in }\cD'(\bR\times\bR^3)
\\
&\Supp(Y_\eps)\subset\bR_+\times\bR^3
\ea
\right.
$$
or, equivalently, $Y_\eps\rstr_{\bR_+\times\bR^3}$ is the solution of the Cauchy problem
$$
\left\{
\ba
{}&\Box_{t,x} Y_\eps=0\,,\quad x\in\bR^3\,,\,\,t>0\,,
\\
&Y_\eps\rstr_{t=0}=0\,,
\\
&\d_tY_\eps\rstr_{t=0}=\chi_\eps\star_x\chi_\eps\,.
\ea
\right.
$$
The classical theory of regularity and the finite speed of propagation for solutions of the wave equation imply that, for each $\eps>0$,
\be\lb{PtesYeps}
Y_\eps\in C^\infty(\bR_+\times\bR^3)\hbox{ and }\Supp(Y_\eps(t,\cdot))\subset B(0,t+2\eps)\,.
\ee

Next we consider the system of (delay) differential equations
\be\lb{DynSystVMReg}
\left\{
\ba
{}&\dot{x}_i(t)=v(\xi_i(t))
\\
&\dot\xi_i(t)=-\frac1{N}\sum_{j=1}^N\grad_x\d_tY_\eps(t,\cdot)\star_xG(x_i(0)-x_j(0))
\\
&-\frac1{N}\sum_{j=1}^N\int_0^t(\grad_x+v(\xi_j(s))\d_t)Y_\eps(t-s,x_i(t)-x_j(s))ds
\\
&-\frac1{N}\sum_{j=1}^N\int_0^tv(\xi_i(t))\wedge(v(\xi_j(s))\wedge\grad_xY_\eps(t-s,x_i(t)-x_j(s)))ds\,.
\\
&\hskip9.8cm {i=1,\ldots,N\,.}
\ea
\right.
\ee
Like (\ref{DynSystVM}), this is a system of $N$ coupled delay differential equations. However, unlike in the case of (\ref{DynSystVM}), the r.h.s. of 
the second equation above is obviously well-defined since $Y_\eps\in C_c^\infty([0,T]\times\bR^3)$ for each $T>0$.

Comparing the original dynamics (\ref{DynSystVM}) with its regularized analogue (\ref{DynSystVMReg}), one notices that, in (\ref{DynSystVM})
the force acting on the $i$-th particle is the electric or electromagnetic force created by the $N-1$ other particles. Self-interaction --- i.e. the action 
of the electromagnetic field created by the $i$-th particle on itself --- is therefore neglected in (\ref{DynSystVM}), unlike in the regularized model
(\ref{DynSystVMReg}). As a matter of fact, there are serious conceptual difficulties with the notion of self-force in classical electrodynamics, which 
are beyond the scope of the present paper. The interested reader is referred to the discussion on pp. 675--676 of \cite{ElsKiesRic09} --- see also 
chapter 16 of \cite{Jackson} and especially \cite{Spohn04} for more on this subject. That the self-interaction force is present in the regularized 
model (\ref{DynSystVMReg}) is only for mathematical convenience (see below).

In fact, if the regularizing parameter is kept fixed, neglecting the self-interaction force in (\ref{DynSystVMReg}) will only produce an error of order
$O(1/N)$ over finite time intervals. 

Indeed, consider instead of (\ref{DynSystVMReg}) the following regularized variant of (\ref{DynSystVM}):
\be\lb{DynSystVMReg2}
\left\{
\ba
{}&\dot{\hat x}_i(t)=v(\hat\xi_i(t))
\\
&\dot{\hat\xi}_i(t)=-\frac1{N-1}\sum_{j=1\atop j\not=i}^N\grad_x\d_tY_\eps(t,\cdot)\star_xG(\hat x_i(0)-\hat x_j(0))
\\
&-\frac1{N-1}\sum_{j=1\atop j\not=i}^N\int_0^t(\grad_x+v(\hat \xi_j(s))\d_t)Y_\eps(t-s,\hat x_i(t)-\hat x_j(s))ds
\\
&-\frac1{N-1}\sum_{j=1\atop j\not=i}^N\int_0^tv(\hat\xi_i(t))\wedge(v(\hat\xi_j(s))\wedge\grad_xY_\eps(t-s,\hat x_i(t)-\hat x_j(s)))ds\,.
\\
&\hskip9.8cm {i=1,\ldots,N\,.}
\ea
\right.
\ee

That the regularized dynamics defined by (\ref{DynSystVMReg}) and (\ref{DynSystVMReg2}) are equivalent on finite time intervals as $N\to+\infty$
is summarized in the following statement.

\begin{Prop}\lb{P-EquivDynSyst12}
Let $\eps,T>0$ and $N\ge 2$. If $(x_i,\xi_i)_{1\le i\le N}$ and $(\hat x_i,\hat\xi_i)_{1\le i\le N}$ are solutions of (\ref{DynSystVMReg}) and of 
(\ref{DynSystVMReg2}) respectively on $[0,T]$, and if 
$$
x_i(0)=\hat x_i(0)\,,\quad\xi_i(0)=\hat\xi_i(0)\quad\hbox{ for }i=1,\ldots,N,
$$
then one has the error estimate
$$
|\hat x_i(t)-x_i(t)|+|\hat\xi_i(t)-\xi_i(t)|=O(1/N)
$$
for all $i=1,\ldots,N$ and all $t\in[0,T]$.
\end{Prop}

\begin{proof}
Observe that, for each $N\ge 2$ and 
$$
\ba
{}&\Big|\frac1{N-1}\sum_{j=1\atop j\not=i}^N\int_0^t(\grad_x+v(\hat \xi_j(s))\d_t)Y_\eps(t-s,\hat x_i(t)-\hat x_j(s))ds
\\
&\qquad\qquad\qquad-\frac1{N}\sum_{j=1}^N\int_0^t(\grad_x+v(\xi_j(s))\d_t)Y_\eps(t-s,x_i(t)-x_j(s))ds\Big|
\\
&\le
\frac1{N(N-1)}\sum_{j=1\atop j\not=i}^N\int_0^t|(\grad_x+v(\hat \xi_j(s))\d_t)Y_\eps(t-s,\hat x_i(t)-\hat x_j(s))|ds
\\
&+
\frac1{N}\sum_{j=1\atop j\not=i}^N\int_0^t\big|(\grad_x+v(\hat \xi_j(s))\d_t)Y_\eps(t-s,\hat x_i(t)-\hat x_j(s))
\\
&\qquad\qquad\qquad\qquad\qquad-(\grad_x+v(\xi_j(s))\d_t)Y_\eps(t-s,x_i(t)-x_j(s))\big|ds
\\
&+\frac1N\int_0^t\big|(\grad_x+v(\xi_i(s))\d_t)Y_\eps(t-s,x_i(t)-x_i(s))\big|ds
\\
&\le
\frac{2t}{N-1}\|\grad_{t,x}Y_\eps\|_{L^\infty([0,T]\times\bR^3)}
\\
&+\frac1{N}\sum_{j=1\atop j\not=i}^N\|\grad_{t,x}^2 Y_\eps\|_{L^\infty([0,T]\times\bR^3)}\int_0^t|\hat x_j(s)-x_j(s)|ds
\\
&+\frac1{N}\sum_{j=1\atop j\not=i}^N\|\d_tY_\eps\|_{L^\infty([0,T]\times\bR^3)}\|\grad_\xi v\|_{L^\infty(\bR^3)}\int_0^t|\hat\xi_j(s)-\xi_j(s)|ds
\\
&+2T\|\grad_{t,x}^2 Y_\eps\|_{L^\infty([0,T]\times\bR^3)}|\hat x_i(t)-x_i(t)|\,.
\\
\ea
$$

The difference between the right-hand sides of the equations for $\dot\xi_i$ in (\ref{DynSystVMReg}) and (\ref{DynSystVMReg2}) involves two
more terms analogous to this one.

Thus, denoting
$$
\cE_i(t):=|\hat x_i(t)-x_i(t)|+|\hat\xi_i(t)-\xi_i(t)|\,,
$$
we see that, for each $i=1,\ldots,N$ and all $t\in[0,T]$, one has
$$
\ba
\cE_i(t)\le C_\eps\left(\int_0^t\cE_i(\tau)d\tau+\frac1N\sum_{j=1}^N\int_0^t\int_0^\tau\cE_i(s)dsd\tau+\frac{4t+2}{N}\right)
\\
=C_\eps\left(\int_0^t\cE_i(\tau)d\tau+\frac1N\sum_{j=1}^N\int_0^t(t-s)\cE_i(s)ds+\frac{4t+2}{N}\right)
\\
\le C_\eps\left(\int_0^t\cE_i(\tau)d\tau+\frac{T}{N}\sum_{j=1}^N\int_0^t\cE_i(s)ds+\frac{4T+2}{N}\right)\,,
\ea
$$
since $\cE_i(0)=0$, where $C_\eps$ is a constant depending on 
$$
\ba
\|\grad_{t,x}Y_\eps\|_{L^\infty([0,T]\times\bR^3)}\,,\,\,\|\grad_{t,x}^2Y_\eps\|_{L^\infty([0,T]\times\bR^3)}\,,\,\,\|G\star_x\grad_xY_\eps\|_{L^\infty([0,T]\times\bR^3)}&
\\
\hbox{ and }\|\grad_\xi v\|_{L^\infty(\bR^3)}&\,.
\ea
$$

Therefore, summing for $i=1,\ldots,N$ both sides of the inequality above shows that
$$
\cE(t):=\frac1N\sum_{i=1}^N\cE_i(t)\le C_\eps\left((1+T)\int_0^t\cE(s)ds+\frac{4T+2}{N}\right)
$$
and we conclude from Gronwall's inequality that
$$
\cE(t)\le C_\eps\frac{4T+2}{N}e^{C_\eps(1+T)t}\,,\quad 0\le t\le T\,.
$$
Substituting this in the inequality satisfied by $\cE_i$, we find that
$$
\ba
\cE_i(t)\le C_\eps\left(\int_0^t\cE_i(\tau)d\tau+T\int_0^t\cE(s)ds+\frac{4T+2}{N}\right) 
\\
\le C_\eps\int_0^t\cE_i(\tau)d\tau+\frac{4T}Ne^{C_\eps(1+T)t}+\frac{2}{N}
\ea
$$
so that, applying again Gronwall's inequality leads to
$$
\cE_i(t)\le\frac2N(1+2Te^{C_\eps(1+T)T})e^{C_\eps t}\,,\quad 0\le t\le T\,,\,\,i=1,\ldots,N\,.
$$
\end{proof}

\section{Dobrushin's estimate}

Let $k$ be a $\bR^n$-valued Radon measure on $\bR^{1+d}\times\bR^{1+d}$ of the form
\be\lb{Form-k}
k(t,z,\tau,\zeta)=m(t,z,\zeta)\de_{\tau=0}+r(t,z,\tau,\zeta)\,,\quad z,\zeta\in\bR^d\,,\,\,s,t\ge 0\,.
\ee
We assume that $r\in C_b(\bR^{1+d}\times\bR^{1+d};\bR^n)$ and $m\in C_b(\bR^{1+d}\times\bR^d;\bR^n)$ satisfy the properties
\be\lb{Causal}
\ba
r(t,z,\tau,\zeta)&=0\quad\hbox{ whenever }\tau\notin[0,t]\hbox{ and }z,\zeta\in\bR^d\,,
\\
m(t,z,\zeta)&=0\quad\hbox{ whenever }t<0\hbox{ and }z,\zeta\in\bR^d\,,
\ea
\ee
together with the Lipschitz condition
\be\lb{Lip}
\ba
|r(t,z,\tau,\zeta)-r(t,z',\tau,\zeta')|+|m(t,z,\zeta_0)-r(t,z',\zeta_0')|
\\
\le\Lip_T(k)(1\wedge|z-z'|+1\wedge|\zeta-\zeta'|+1\wedge|\zeta_0-\zeta_0'|)
\ea
\ee
for all $t,t',\tau,\tau'\in[0,T]$ and all $z,z',\zeta,\zeta',\zeta_0,\zeta_0'\in\bR^d$, where the notation $a\wedge b$ designates $\min(a,b)$ for all
$a,b\in\bR$.

We consider the integral operator with kernel $k$, i.e.
$$
K\mu(t,z)=\int_{\bR^{1+d}}r(t,z,\tau,\zeta)\mu(\tau,d\zeta)d\tau+\int_{\bR^d}m(t,z,\zeta)\mu(0,d\zeta)
$$
for each $\mu\in C(\bR_+;w-\cP(\bR^d))$, where $w-\cP(\bR^d)$ designates the set of Borel probability measures on $\bR^d$ equipped with its
weak topology.

For any given Borel probability measure $\rho^{in}$ on $\bR^d$, consider $Z\equiv Z(t,z;\rho^{in})$, the solution of the mean-field integro-differential 
equation
\be\lb{MFEq}
\left\{
\begin{array}{l}
\dot{Z}=K\rho(t,Z)\,,
\\
\rho(t,\cdot)=Z(t,\cdot;\rho^{in})\#\rho^{in}\,,
\\
Z(0,z;\rho^{in})=z\,.
\end{array}
\right.
\ee
(If $\Phi:\,\bR^d\mapsto\bR^d$ is a Borel transformation and $\rho$ a Borel probability measure on $\bR^d$, the notation $\Phi\#\rho$ designates 
the image measure of $\rho$ under $\Phi$, defined by the formula $\Phi\#\rho(A)=\rho(\Phi^{-1}(A))$ for each Borel set $A\subset\bR^d$.)

\begin{Prop}\lb{P-ExistUniqMFEq}
For each Borel probability measure $\rho^{in}$ on $\bR^d$, the mean-field integro-differential equation (\ref{MFEq}) has a unique global 
solution $Z\in C(\bR_+\times\bR^d;\bR^d)$ such that $t\mapsto Z(t,z)$ is continuously differentiable on $\bR_+$ for each $z\in\bR^d$.
\end{Prop}

\begin{proof}
To any bounded map $Z$ belonging to $C(\bR_+\times\bR^d;\bR^d)$, we associate the map $\tilde Z$ in $C(\bR_+\times\bR^d;\bR^d)$ 
defined by
$$
\ba
\tilde Z(t,z)&=z+\int_0^tK(Z(\cdot,\cdot)\#\rho^{in})(\tau,Z(\tau,z))d\tau
\\
&=z+\int_0^t\int_0^s\int_{\bR^d}r(s,Z(s,z),\tau,Z(\tau,\zeta))\rho^{in}(d\zeta)d\tau ds
\\
&\qquad\qquad\qquad\quad+\int_0^t\int_{\bR^d}m(s,Z(s,z),\zeta)\rho^{in}(d\zeta)ds\,.
\ea
$$
Let $Z_1,Z_2\in C(\bR_+\times\bR^d;\bR^d)$ and denote $\tilde Z_1,\tilde Z_2$ the maps so defined in terms of $Z_1,Z_2$. Then, by the 
Lipschitz condition (\ref{Lip})
$$
\ba
|\tilde Z_1(t,z)-&\tilde Z_2(t,z)|
\\
&\le\Lip_T(k)\int_0^t\int_0^s\int_{\bR^d}|Z_1(s,z)-Z_2(s,z)|\rho^{in}(d\zeta)d\tau ds
\\
&+\Lip_T(k)\int_0^t\int_0^s\int_{\bR^d}|Z_1(\tau,\zeta)-Z_2(\tau,\zeta)|\rho^{in}(d\zeta)d\tau ds
\\
&+\Lip_T(k)\int_0^t\int_{\bR^d}|Z_1(s,z)-Z_2(s,z)|\rho^{in}(d\zeta)ds\,,
\ea
$$
for all $t\in[0,T]$. Denote
$$
\left\{
\ba
\cE(t):=\sup_{z\in\bR^N}|Z_1(t,z)-Z_2(t,z)|\,,
\\
\tilde\cE(t):=\sup_{z\in\bR^N}|\tilde Z_1(t,z)-\tilde Z_2(t,z)|\,,
\ea
\right.
$$
for all $t\ge 0$. Then, for all $t\in[0,T]$, one has
\be\lb{E-Etilde}
\ba
\tilde\cE(t)\le\Lip_T(k)\left(\int_0^t(1+s)\cE(s)ds+\int_0^t\int_0^s\cE(\tau)d\tau ds\right)
\\
=\Lip_T(k)\left(\int_0^t(1+s)\cE(s)ds+\int_0^t(t-\tau)\cE(\tau)d\tau\right)
\\
=\Lip_T(k)(1+t)\int_0^t\cE(s)ds\,.
\ea
\ee

With this estimate, we construct the solution by the usual Picard iteration procedure. Define recursively a sequence of maps $(Z^n)_{n\ge 0}$
by 
$$
\left\{
\ba
{}&Z^{n+1}(t,z):=z+\int_0^tK(Z^{n}(\cdot,\cdot)\#\rho^{in})(\tau,Z^n(\tau,z))d\tau\,,\quad t\ge 0\,,
\\
&Z^0(t,z):=z\,.
\ea
\right.
$$
Applying (\ref{E-Etilde}) iteratively with the notation 
$$
\cE^n(t)=\sup_{z\in\bR^d}|Z^{n+1}(t,z)-Z^{n}(t,z)|
$$
shows that, for all $t\in[0,T]$, one has
\be\lb{Picard}
\ba
\cE^n(t)&\le\Lip_T(k)(1+t)\int_0^t\cE^{n-1}(t_1)dt_1
\\
&\le\Lip_T(k)^2(1+t)\int_0^t(1+t_1)\int_0^{t_1}\cE^{n-2}(t_2)dt_2dt_1
\\
&\le\Lip_T(k)^2(1+t)^2\int_0^t\int_0^{t_1}\cE^{n-2}(t_2)dt_2dt_1
\\
&\ldots
\\
&\le\Lip_T(k)^n(1+t)^n\frac{t^n}{n!}\sup_{0\le t_n\le t}\cE^{0}(t_n)
\ea
\ee
since
$$
\int\indc_{0\le t_n\le t_{n-1}\le\ldots\le t_1\le t}dt_ndt_{n-1}\ldots dt_1=\frac{t^n}{n!}\,.
$$

Since
$$
\sum_{n\ge 0}\frac{\Lip_T(k)^n(1+t)^nt^n}{n!}<+\infty
$$
for each $t>0$, we conclude that the sequence $Z^n$ converges uniformly on $[0,T]\times\bR^d$ for each $T>0$ towards a solution of the
integral equation
\be\lb{IntEq}
Z(t,z)=z+\int_0^tK(Z(\cdot,\cdot)\#\rho^{in})(\tau,Z(\tau,z))d\tau
\ee
defined for each $t\ge 0$ and such that $Z\in C(\bR_+\times\bR^d;\bR^d)$.

If $Z_1$ and $Z_2$ are two solutions of that integral equation, denoting as above
$$
\cE(t)=\sup_{z\in\bR^d}|Z_1(t,z)-Z_2(t,z)|\,,
$$
we conclude from (\ref{E-Etilde}) that
$$
\cE(t)\le\Lip_T(k)(1+t)\int_0^t\cE(s)ds\,,
$$
and, following the argument in (\ref{Picard}), that
$$
\sup_{0\le t\le T}\cE(t) \le\Lip_T(k)^n(1+T)^n\frac{T^n}{n!}\sup_{0\le t\le T}\cE(t)\,.
$$
Letting $n\to+\infty$ while keeping $T>0$ fixed shows that $\cE(t)=0$ for each $t\in[0,T]$; and since $T>0$ is arbitrary, $Z_1=Z_2$ on 
$\bR_+\times\bR^d$. This establishes the uniqueness of the solution $Z$ of the integral equation (\ref{IntEq}). 

Finally, since the integral kernels $r$ and $m$ are bounded continuous functions of all their arguments while $\rho^{in}$ is a probability measure 
on $\bR^d$,  applying the dominated convergence theorem shows that $t\mapsto K(Z(\cdot,\cdot)\#\rho^{in})(t,z)$ is continuous for all $z\in\bR^d$, 
so that $t\mapsto Z(t,z)$ is continuously differentiable in $t$ for all $z\in\bR^d$ and satisfies the differential form (\ref{MFEq}) of the integral equation 
(\ref{IntEq}).
\end{proof}

\smallskip
Once the existence and uniqueness of the solution of the mean-field equation (\ref{MFEq}) is known, the next natural question is whether
its dependence in the initial probability measure $\rho^{in}$ is continuous. The right topology in which to measure the proximity of two 
initial probability densities $\rho_1^{in},\rho_2^{in}$ is obviously that of the weak convergence of probability measures on $\bR^d$. An
extremely convenient way to obtain quantitative information in that topology is to use the Monge-Kantorovich-Rubinstein distance, whose
definition we recall below.

\begin{Def}
Let $\cP(\bR^d)$ be the set of Borel probability measures on $\bR^d$. The Monge-Kantorovich-Rubinstein distance on $\cP(\bR^d)$ is
defined by the formula
$$
\Dist_{MKR}(\mu,\nu)=\inf_{\pi\in\Pi(\mu,\nu)}\iint_{\bR^d\times\bR^d}1\wedge|x-y|\pi(dxdy)
$$
where $\Pi(\mu,\nu)$ is the set of Borel probability measures on $\bR^d\times\bR^d$ whose marginals are $\mu$ and $\nu$, i.e.
$$
\iint_{\bR^d\times\bR^d}\phi(x)\psi(y)\pi(dxdy)=\int_{\bR^d}\phi(x)\mu(dx)\int_{\bR^d}\psi(y)\nu(dy)
$$
for all bounded continuous functions $\phi,\psi$ defined on $\bR^d$. 
\end{Def}

This distance is also called the Wasserstein distance --- although it was considered more than ten years earlier by Kantorovich and 
Rubinstein, and appears in the work of Monge.

Our (slight) generalization of the Dobrushin estimate presented in \cite{Dobrushin} is stated in the following proposition.

\begin{Prop}\lb{P-DobrushIneq}
Let $\rho_1^{in}$ and $\rho_2^{in}$ be two Borel probability measures on $\bR^d$, and let $Z_1,Z_2$ be the corresponding solutions of 
(\ref{MFEq}), denoted 
$$
Z_j(t,z):=Z(t,z;\rho_j^{in})\hbox{ for }j=1,2,\quad t\ge 0,\hbox{ and }z\in\bR^d,
$$
whose existence and uniqueness is established by Proposition \ref{P-ExistUniqMFEq}, and 
$$
\rho_j(t,\cdot):=Z_j(t,\cdot)\#\rho_j^{in}\,,\quad\hbox{ for } j=1,2\,.
$$
Then, for each $t\ge 0$,
$$
\Dist_{MKR}(\rho_1(t,\cdot),\rho_2(t,\cdot))\le(1+t\Lip_t(k))e^{t^2\Lip_t(k)}\Dist_{MKR}(\rho_1^{in},\rho_2^{in})\,.
$$
\end{Prop}

The main difference between this result and Proposition 4 in \cite{Dobrushin} is that the proposition above applies to an integro-differential
equation of the form (\ref{MFEq}) satisfying the causality assumption (\ref{Causal}), whereas the differential system considered in Dobrushin's
original contribution \cite{Dobrushin} did not involve memory effects. This is only natural since Dobrushin was interested in the Vlasov-Poisson 
equation, a non-relativistic model in which the force field is created by the instantaneous distribution of charges. 

There are other minor differences: for instance, unlike Dobrushin's the estimate obtained here is local in time and uses the Gronwall 
inequality; besides, the integral operator $K$ in (\ref{MFEq}) considered in the proposition above is not necessarily a convolution operator as 
in \cite{Dobrushin}. But the argument is mostly the same as in \cite{Dobrushin}, and we give it below only for the sake of being complete.

\begin{proof}
One has
$$
\ba
K\rho_1&(t,Z_1(t,z_1))-K\rho_2(t,Z_2(t,z_2))
\\
&=
\int_0^t\int_{\bR^d}r(t,Z_1(t,z_1),\tau,\zeta_1)\rho_1(\tau,d\zeta_1)d\tau
\\
&+\int_{\bR^d}m(t,Z_1(t,z_1),\zeta_1)\rho_1^{in}(d\zeta_1)
\\
&-
\int_0^t\int_{\bR^d}r(t,Z_2(t,z_2),\tau,\zeta_2)\rho_2(\tau,d\zeta_2)d\tau
\\
&-\int_{\bR^d}m(t,Z_2(t,z_2),\zeta_2)\rho_2^{in}(d\zeta_2)
\\
&=
\int_0^t\int_{\bR^d}r(t,Z_1(t,z_1),\tau,Z_1(\tau,\zeta_1))\rho_1^{in}(d\zeta_1)d\tau
\\
&+\int_{\bR^d}m(t,Z_1(t,z_1),\zeta_1)\rho_1^{in}(d\zeta_1)
\\
&-
\int_0^t\int_{\bR^d}r(t,Z_2(t,z_2),\tau,Z_2(\tau,\zeta_2))\rho_2^{in}(d\zeta_2)d\tau
\\
&-\int_{\bR^d}m(t,Z_2(t,z_2),\zeta_2)\rho_2^{in}(d\zeta_2)
\ea
$$
for each $t\ge 0$ and $z_1,z_2\in\bR^d$. In the second equality above, we have used the identity
\be\lb{IntChVar}
\int_{\bR^d}\phi(Z_j(\tau,\zeta_j))\rho_j^{in}(d\zeta_j)=\int_{\bR^d}\phi(\zeta_j)\rho_j(\tau,d\zeta_j)
\ee
for all bounded $\phi$, Borel measurable on $\bR^d$, since $\rho_j(\tau,\cdot)=Z_j(\tau,\cdot)\#\rho_j^{in}$ --- see the 
definition of the image of a probability measure under a Borel map following (\ref{MFEq}), that is equivalent to (\ref{IntChVar}) with $\phi=\indc_A$.

Pick $\pi^{in}$ to be any probability measure on $\bR^d\times\bR^d$ with $\rho_1^{in}$ and $\rho_2^{in}$ as marginals: one has, for each 
$t\ge 0$ and $z_1,z_2\in\bR^d$
$$
\ba
K\rho_1(t,Z_1(t,z_1))-&K\rho_2(t,Z_2(t,z_2))
\\
=
\int_0^t\iint_{\bR^d\times\bR^d}&(r(t,Z_1(t,z_1),\tau,Z_1(\tau,\zeta_1))
\\
&-r(t,Z_2(t,z_2),\tau,Z_2(\tau,\zeta_2)))\pi^{in}(d\zeta_1,d\zeta_2)
\\
+
\iint_{\bR^d\times\bR^d}&(m(t,Z_1(t,z_1),\zeta_1)-m(t,Z_2(t,z_2),\zeta_2))\pi^{in}(d\zeta_1,d\zeta_2)\,.
\ea
$$
Hence, as a consequence of (\ref{Lip}), for each $t\in[0,T]$ and $z_1,z_2\in\bR^d$
$$
\ba
|K&\rho_1(t,Z_1(t,z_1))-K\rho_2(t,Z_2(t,z_2))|
\\
&\le\Lip_T(k)\left(t1\wedge|Z_1(t,z_1)-Z_2(t,z_2)|+\int_0^t\cD(\tau)d\tau+\cD(0)\right)\,,
\ea
$$
with the notation
$$
\cD(\tau):=\iint_{\bR^d\times\bR^d}1\wedge|Z_1(\tau,\zeta_1)-Z_2(\tau,\zeta_2)|\pi^{in}(d\zeta_1,d\zeta_2)\,.
$$

Therefore, for each $t\in[0,T]$ and each $z_1,z_2\in\bR^d$, one has
$$
\ba
{}&|Z_1(t,z_1)-Z_2(t,z_2)|
\\
&\le|z_1-z_2|+\int_0^t|K\rho_1(s,Z_1(s,z_1))-K\rho_2(s,Z_2(s,z_2))|ds
\\
&\le|z_1-z_2|+\Lip_T(k)t\cD(0)
\\
&+\!\Lip_T(k)\left(\int_0^ts1\wedge|Z_1(s,z_1)\!-\!Z_2(s,z_2)|ds\!+\!\!\int_0^t\int_0^s\cD(\tau)d\tau ds\right)\,.
\ea
$$
Integrating with the measure $\pi^{in}(dz_1dz_2)$ both sides of the resulting inequality for $1\wedge|Z_1(t,z_1)-Z_2(t,z_2)|$, we arrive at
$$
\ba
\cD(t)&\le(1+t\Lip_T(k))\cD(0)
\\
&+\Lip_T(k)\left(\int_0^ts\cD(s)ds+\int_0^t\int_0^s\cD(\tau)d\tau ds\right)
\\
&=(1+t\Lip_T(k))\cD(0)
\\
&+\Lip_T(k)\left(\int_0^ts\cD(s)ds+\int_0^t(t-s)\cD(s)ds\right)
\\
&=(1+t\Lip_T(k))\cD(0)+\Lip_T(k)t\int_0^t\cD(s)ds\,.
\ea
$$
This implies the estimate
$$
\cD(t)\le\cD(0)\Psi(t)\,,\quad t\ge 0\,,
$$
where
$$
\Psi(t)=(1+t\Lip_t(k))e^{t^2\Lip_t(k)}\,,\quad t\ge 0\,.
$$

Now
$$
\cD(t)=\iint_{\bR^d\times\bR^d}1\wedge|\zeta_1-\zeta_2|\pi(t;d\zeta_1d\zeta_2)
$$
where
$$
\pi(t,\cdot):=(Z_1(t,\cdot),Z_2(t,\cdot))\#\pi^{in}
$$
is the image under the map $(z_1,z_2)\mapsto(Z_1(t,z_1),Z_2(t,z_2))$ of the measure $\pi^{in}$ . Obviously 
$$
\pi^{in}\in\Pi(\rho_1^{in},\rho_2^{in})\Rightarrow\pi(t,\cdot)\in\Pi(\rho_1(t,\cdot),\rho_2(t,\cdot))
$$
for each $t\ge 0$, so that
$$
\Dist_{MKR}(\rho_1(t,\cdot),\rho_2(t,\cdot))\le\Psi(t)\cD(0)\,.
$$
Minimizing the right-hand side of this inequality as $\pi^{in}$ runs through $\Pi(\rho_1^{in},\rho_2^{in})$ leads to the desired estimate.
\end{proof}

\section{Application to the regularized Vlasov-Maxwell dynamics}

Let $\eps>0$ be a fixed regularization parameter, and consider the regularized Vlasov-Maxwell system (\ref{VMScalReg}). Henceforth we
assume that $f^{in}$ is a probability density on $\bR^3\times\bR^3$. Applying the method of characteristics to the transport equation governing 
$f_\eps$ shows that
\be\lb{MethChareps}
f_\eps(t,X_\eps(t,x,v;f^{in}),\Xi_\eps(t,x,v;f^{in}))=f^{in}(x,v)
\ee
where $t\mapsto(X_\eps,\Xi_\eps)(t,\cdot,\cdot;f^{in})$ is the solution of
\be\lb{CharFieldeps}
\left\{
\ba
\dot{X}_\eps&=v(\Xi_\eps)
\\
\dot{\Xi}_\eps&=E_\eps(t,X_\eps)+v(\Xi_\eps)\wedge B_\eps(t,X_\eps)
\\
&=F_\eps[f_\eps(t,\cdot,\cdot)](X_\eps,\Xi_\eps)\,,
\ea
\right.
\ee
with the notation
\be\lb{DefFeps}
\ba
F_\eps[f]:=&-\int_{\bR^3}\d_t\grad_xY_\eps(t,\cdot)\star_xG\star_xf(0,\cdot,\eta)d\eta
\\
&-\int_{\bR^3}(\grad_x+v(\eta)\d_t)Y_\eps\star_{t,x}f(\cdot,\cdot,\eta)d\eta
\\
&-\int_{\bR^3}v(\xi)\wedge(v(\eta)\wedge\grad_xY_\eps\star_{t,x}f(\cdot,\cdot,\eta))d\eta\,.
\ea
\ee

Since $\Div_{x,\xi}(v(\xi),E_\eps(t,x)+v(\xi)\wedge B_\eps(t,x))=0$, the flow $(X_\eps,\Xi_\eps)$ leaves the Lebesgue measure $dxd\xi$ invariant so 
that (\ref{MethChareps}) can be recast as
\be\lb{MethChareps2}
f_\eps(t,\cdot,\cdot)dxd\xi=(X_\eps,\Xi_\eps)(t,\cdot,\cdot;f^{in})\#f^{in}dxd\xi\,.
\ee

According to (\ref{DefFeps}), the vector field $(v(\xi),E_\eps(t,x)+v(\xi)\wedge B_\eps(t,x))$ can be expressed by an integral operator with kernel of 
the form (\ref{Form-k}) satisfying the causality and Lipschitz conditions (\ref{Causal}) and (\ref{Lip}):
\be\lb{DefKeps}
\ba
(v(\xi),E_\eps(t,x)&+v(\xi)\wedge B_\eps(t,x))=K_\eps((X_\eps,\Xi_\eps)(t,\cdot,\cdot;f^{in})\# f^{in})(t,x,\xi)
\\
&=\int_0^t\iint_{\bR^3\times\bR^3}r_\eps(t,x,\xi,\tau,y,\eta)f_\eps(\tau,y,\eta)dyd\eta d\tau
\\
&+\iint_{\bR^3\times\bR^3}m_\eps(t,x,\xi,y,\eta)f^{in}(y,\eta)dyd\eta\,.
\ea
\ee
In this formula, $r_\eps=(r^1_\eps,r^2_\eps)$ and $m_\eps=(m^1_\eps,m^2_\eps)$, where $r^1_\eps$, $r^2_\eps$, $m^1_\eps$ and $m^2_\eps$ 
are given by 
\be\lb{Def-rmeps}
\ba
r^1_\eps(t,x,\xi,s,y,\eta)&=0\,,
\\
m^1_\eps(t,x,\xi,y,\eta)&=\indc_{t\ge 0}v(\xi)\,,
\\
r^2_\eps(t,x,\xi,s,y,\eta)&=-(\grad_x+v(\eta)\d_t)Y_\eps(t-s,x-y)
\\
&-v(\xi)\wedge(v(\eta)\wedge\grad_xY_\eps(t-s,x-y))\,,
\\
m^2_\eps(t,x,\xi,y,\eta)&=-\d_t\grad_xY_\eps\star_xG(t,x-y)\,.
\ea
\ee

Therefore the characteristic flow $Z_\eps:=(X_\eps,\Xi_\eps)$ of the regularized Vlasov-Maxwell system (\ref{VMScalReg}) satisfies a mean-field
integro-differential system of the form (\ref{MFEq}), where the integral operator $K_\eps$ is defined by (\ref{DefKeps})-(\ref{Def-rmeps}). That the
integral kernel of $K_\eps$ satisfies (\ref{Causal}) is obvious since $Y_\eps(t,x)=0$ for all $x\in\bR^3$ if $t<0$. That it also satisfies the Lipschitz
condition (\ref{Lip}) is equally obvious in view of (\ref{PtesYeps}). Henceforth, we denote 
\be\lb{DefLTeps}
L(T,\eps)=\Lip_T(k_\eps)
\ee
where $k_\eps$ is the integral kernel of the operator $K_\eps$ defined above.

Bringing together the results obtained in the previous sections, we finally obtain the regularized Vlasov-Maxwell system (\ref{VMReg})  ---
or equivalently (\ref{VMScalReg}) as the mean field limit of the $N$-particle system (\ref{DynSystVMReg}) as $N\to+\infty$.

\begin{Thm}\lb{T-MFLim}
Let $\eps>0$ and $f^{in}\in L^\infty(\bR^3\times\bR^3)$ be such that 
$$
f^{in}\ge 0\hbox{ a.e. on }\bR^3\times\bR^3\,,\quad\Supp(f^{in})\hbox{ is compact,}
$$
and
$$
\iint_{\bR^3\times\bR^3}f^{in}(x,\xi)dxd\xi=1\,.
$$
(In other words, $f^{in}$ is an essentially bounded probability distribution on $\bR^3\times\bR^3$.)

Let $(f_\eps,u_\eps)$ be the solution of (\ref{VMScalReg}) with initial data (\ref{CondInVMScalReg}) --- or, equivalently $(f_\eps,E_\eps,B_\eps)$ 
the solution of (\ref{VMReg}) with initial data (\ref{CondInReg}).

On the other hand, let $(x^{in}_{i,N},\xi^{in}_{i,N})_{1\le i\le N}$ be a sequence (indexed by $N$) of $2N$-tuples in $\bR^3$ such that the initial
empirical measure of the $N$-particle system
$$
f^{in}_N:=\frac1N\sum_{i=1}^N\de_{x^{in}_{i,N}}\otimes\de_{\xi^{in}_{i,N}}\to f^{in}dxd\xi\hbox{ weakly }
$$
as $N\to+\infty$. Then

\smallskip
\noindent
a) for each $N\ge 1$, the system (\ref{DynSystVMReg}) has a unique solution defined for all $t\ge 0$ and denoted 
$t\mapsto(x_{i,N}(t),\xi_{i,N}(t))_{1\le i\le N}\,$.

\smallskip
Denote
$$
f_{N,\eps}(t,\cdot,\cdot):=\frac1N\sum_{i=1}^N\de_{x_{i,N}(t)}\otimes\de_{\xi_{i,N}(t)}
$$
the $N$-particle empirical measure  at time $t$ for each $t\ge 0$. Then

\noindent
b) for each $T>0$,
$$
f_{N,\eps}(t,\cdot,\cdot)\to f_\eps(t,x,\xi)dxd\xi\hbox{ weakly in }\cP(\bR^3\times\bR^3)
$$
uniformly in $t\in[0,T]$ as $N\to+\infty$; more precisely
$$
\ba
\sup_{t\in[0,T]}\Dist_{MKR}\left(f_\eps(t,\cdot,\cdot),f_{N,\eps}(t,\cdot,\cdot)\right)
\le
C(T,\eps)\Dist_{MKR}\left(f^{in},f^{in}_N\right)
\ea
$$
where $C(T,\eps)=(1+TL(T,\eps))e^{T^2L(T,\eps)}$ with $L(T,\eps)$ as in (\ref{DefLTeps}).
\end{Thm}

\begin{proof}
As we have explained before stating Theorem \ref{T-MFLim}, the regularized Vlasov-Maxwell system (\ref{VMScalReg}) with initial data 
(\ref{CondInVMScalReg}), or equivalently (\ref{VMReg}) with initial data (\ref{CondInReg}) is a special case of (\ref{MFEq}) with integral
operator $K_\eps$ defined by (\ref{DefKeps})-(\ref{Def-rmeps}).

By Proposition \ref{P-ExistUniqMFEq}, the system (\ref{MFEq}) with integral operator $K_\eps$ has a unique solution 
$$
(X_\eps,\Xi_\eps)\equiv(X_\eps,\Xi_\eps)(t,x,\xi;f^{in})
$$
for any initial  Borel probability measure $f^{in}$ on $\bR^3\times\bR^3$.

On the other hand, if the map $t\mapsto(x_{i,N}(t),\xi_{i,N}(t))_{1\le i\le N}$ is $C^1$ on $\bR_+$, a straightforward computation shows that it is a 
solution of (\ref{DynSystVMReg}) with initial data
$$
(x_{i,N}(0),\xi_{i,N}(0))_{1\le i\le N}=(x_{i,N}^{in},\xi_{i,N}^{in})_{1\le i\le N}
$$
if and only if 
$$
x_{i,N}(t)=X_\eps(t,x_{i,N}^{in},\xi_{i,N}^{in};f^{in}_N)\,,\quad\xi_{i,N}(t)=\Xi_\eps(t,x_{i,N}^{in},\xi_{i,N}^{in};f^{in}_N)\,.
$$
This is in turn equivalent to the fact that the empirical measure
$$
f_{N,\eps}(t,\cdot,\cdot)=\frac1N\sum_{i=1}^N\de_{x_{i,N}(t)}\otimes\de_{\xi_{i,N}(t)}
$$
satisfies
$$
f_{N,\eps}(t,\cdot,\cdot)=(X_\eps,\Xi_\eps)(t,\cdot,\cdot;f^{in}_N)\#f^{in}_N\,.
$$
Therefore the existence and uniqueness result in Proposition \ref{P-ExistUniqMFEq} translates into statement a) in the theorem.

Comparing this last equality with (\ref{MethChareps2}) and using the Dobrushin inequality in Proposition \ref{P-DobrushIneq}, we
arrive at b). (We recall that the distance $\Dist_{MKR}$ metrizes the topology of weak convergence on the set $\cP(\bR^3\times\bR^3)$
of probability measures on $\bR^3\times\bR^3$.)
\end{proof}

Theorem \ref{T-MFLim} can be viewed as a quantitative statement about the continuity in the weak topology of probability measures of the 
solution of the regularized Vlasov-Maxwell system (\ref{VMReg}) in terms of its initial data. Its corollary stated below explains how this information 
can be translated into uniform convergence of the $N$-particle regularized electromagnetic field to the one created by the solution of (\ref{VMReg}).

Let $(E_\eps,B_\eps)$ be the electromagnetic field associated to the solution of (\ref{VMReg}) with initial data $f^{in}$ and let 
$(E_{N,\eps},B_{N,\eps})$ be the solution of the regularized Maxwell system 
\be\lb{MaxReg}
\left\{
\begin{array}{l}
\Div_xB_{N,\eps}=0\,,
\\
\d_tB_{N,\eps}+\Rot_xE_{N,\eps}=0\,,
\\
\Div_xE_{N,\eps}=\chi_\eps\star_x\chi_\eps\star_x\rho_{f_{N,\eps}}\,,
\\
\d_tE_{N,\eps}-\Rot_xB_{N,\eps}=-\chi_\eps\star_x\chi_\eps\star_xj_{f_{N,\eps}}\,,
\end{array}
\right.
\ee
with initial data
\be\lb{CondInMaxReg}
E_{N,\eps}\rstr_{t=0}=-\grad_x\chi_\eps\star_x\chi_\eps\star_xG\star_x\rho_{f^{in}_N}\,,\quad B_\eps\rstr_{t=0}=0\,,
\ee
In other words, $(E_{N,\eps},B_{N,\eps})$ is the regularized electromagnetic field created by the $N$-particle system.

\begin{Cor}\lb{C-MFLim}
Under the same assumptions and with the same notations as in Theorem \ref{T-MFLim}

\noindent
a) for each $T>0$, one has
$$
\Dist_{MKR}\left(\rho_{f_\eps}(t,\cdot),\rho_{f_{N,\eps}}(t,\cdot)\right)\le C(T,\eps)\Dist_{MKR}\left(f^{in},f^{in}_N\right)
$$
and
$$
\|j_{f_\eps}(t,\cdot)-j_{f_{N,\eps}}(t,\cdot)\|_{W^{-1,1}(\bR^3)}\le 4C(T,\eps)\Dist_{MKR}\left(f^{in},f^{in}_N\right)
$$
for all $t\in[0,T]$ and $N\ge 1$ --- so that in particular 
$$
\frac1N\sum_{i=1}^N\de_{x_{i,N}(t)}\to\rho_{f_\eps}(t,\cdot)
	\quad\hbox{ and }\quad\frac1N\sum_{i=1}^Nv(\xi_{i,N}(t))\de_{x_{i,N}(t)}\to j_{f_\eps}(t,\cdot)
$$
uniformly in $t\in[0,T]$ as $N\to+\infty$, where the first convergence holds in the weak topology of $\cP(\bR^3)$, while the second holds in 
the weak topology of bounded Radon measures on $\bR^3$;

b) for each $T>0$, $(E_{N,\eps},B_{N,\eps})$ converges to $(E_\eps,B_\eps)$ uniformly in $(t,x)\in[0,T]\times\bR^3$ with the following estimates
$$
\|B_\eps-B_{N,\eps}\|_{L^\infty([0,T]\times\bR^3)}\le C'(T,\eps)\Dist_{MKR}(f^{in},f^{in}_N)
$$
and
$$
\|E_\eps-E_{N,\eps}\|_{L^\infty([0,T]\times\bR^3)}\le C''(T,\eps)\Dist_{MKR}(f^{in},f^{in}_N)\,,
$$
where
$$
C'(T,\eps)=4TC(T,\eps)\|\grad_xY_\eps\|_{L^\infty([0,T];W^{1,\infty}(\bR^3))}
$$
and
$$
\ba
C''(T,\eps)=4TC(T,\eps)\|\grad_{t,x}Y_\eps\|_{L^\infty([0,T];W^{1,\infty}(\bR^3))}&
\\
+\|d_t\grad_xY_\eps\star_xG\|_{L^\infty([0,T];W^{1,\infty}(\bR^3))}&\,,
\ea
$$
and where $C(T,\eps)$ is as in Theorem \ref{T-MFLim}.
\end{Cor}

\begin{proof}
For the first statement in a), pick any bounded and Lipschitz-continuous function $\phi$ on $\bR^3$ and observe that
$$
\ba
\left|\int_{\bR^3}(\rho_{f_\eps}(t,z)-\rho_{f_{N,\eps}}(t,z))\phi(z)dx\right|
\\
=\left|\iint_{\bR^3\times\bR^3}(\phi(x)-\phi(y))\pi(t,dxd\xi dyd\eta)\right|
\\
\le 2\|\phi\|_{W^{1,\infty}}\iint_{\bR^3\times\bR^3\times\bR^3\times\bR^3}1\wedge|x-y|\pi(t,dxd\xi dyd\eta)
\ea
$$
where $\pi(t,\cdot)$ is any probability measure in $\Pi(f_\eps(t,\cdot),f_{N,\eps}(t,\cdot))$. Minimizing the r.h.s. in $\pi(t,\cdot)$ while maximizing
the l.h.s. in $\phi$ such that $\Lip(\phi)\le 1$ and applying the Kantorovich duality theorem (see formula (5.11) on p. 60 in \cite{VillaniTOT}, or
Theorem 1.3 on p. 19 in \cite{VillaniTOTold} and Remark 7.5 (i) on p. 207 in that same reference) leads to 
$$
\Dist_{MKR}(\rho_{f_\eps}(t,\cdot),\rho_{f_{N,\eps}}(t,\cdot))\le\Dist_{MKR}(f_\eps(t,\cdot),f_{N,\eps}(t,\cdot))\,,
$$ 
and one concludes with Theorem \ref{T-MFLim} b).

For the second statement, proceed in the same way, replacing the test function $\phi$ with $\psi(x)\cdot v(\xi)$, assuming that 
$\psi\in W^{1,\infty}(\bR^3;\bR^3)$. Thus
$$
\ba
\left|\int_{\bR^3}(j_{f_\eps}(t,z)-j_{f_{N,\eps}}(t,z))\cdot\psi(z)dx\right|
\\
=\left|\iint_{\bR^3\times\bR^3}(\psi(x)\cdot v(\xi)-\psi(y)\cdot v(\eta))\pi(t,dxd\xi dyd\eta)\right|
\\
\le(2\|\phi\|_{W^{1,\infty}}+2\|\psi\|_{L^\infty})
\\
\times\iint_{\bR^3\times\bR^3\times\bR^3\times\bR^3}(1\wedge|x-y|+1\wedge|\xi-\eta|)\pi(t,dxd\xi dyd\eta)\,.
\ea
$$
Minimizing the r.h.s. in $\pi(t,\cdot)$ while maximizing the l.h.s. in $\psi$ such that $\|\psi\|_{W^{1,\infty}(\bR^3)}\le 1$ leads to 
$$
\|j_{f_\eps}(t,\cdot)-j_{f_{N,\eps}}(t,\cdot)\|_{W^{-1,1}(\bR^3)}\le 4\Dist_{MKR}(f_\eps(t,\cdot),f_{N,\eps}(t,\cdot))
$$ 
and one concludes again with Theorem \ref{T-MFLim} b).

Now for b). Observe that
$$
B_\eps(t,x)-B_{N,\eps}(t,x)=\int_0^t\int_{\bR^3}\grad_xY_\eps(t-s,x-y)\wedge(j_{f_\eps}(s,y)-j_{f_{N,\eps}}(s,y))dy
$$
so that, for all $x\in\bR^3$, one has
$$
\ba
|B_\eps(t,x)-B_{N,\eps}(t,x)|\le&t\sup_{0\le s\le t}\|\grad_xY_\eps(s,\cdot)\|_{W^{1,\infty}(\bR^3)}
\\
&\times\sup_{0\le s\le t}\|j_{f_\eps}(t,\cdot)-j_{f_{N,\eps}}(t,\cdot)\|_{W^{-1,1}(\bR^3)}
\ea
$$
and one concludes by a). The estimate for the electric field is obtained in a similar way, from the formula
$$
\ba
E_\eps(t,x)&-E_{N,\eps}(t,x)
\\
&=-\int_0^t\int_{\bR^3}\d_tY_\eps(t-s,x-y)(j_{f_\eps}(s,y)-j_{f_{N,\eps}}(s,y))dy
\\
&-\int_0^t\int_{\bR^3}\grad_xY_\eps(t-s,x-y)(\rho_{f_\eps}(s,y)-\rho_{f_{N,\eps}}(s,y))dy
\\
&-\int_{\bR^3}(\grad_x\d_tY_\eps(t,\cdot)\star_xG)(x-y)(\rho_{f^{in}}(y)-\rho_{f^{in}_{N}}(y))dy\,.
\ea
$$
\end{proof}


\section{Final remarks and conclusion}


The discussion above leaves aside several additional questions about this regularized dynamics.

First, although we insisted on choosing Rein's regularization procedure because the ``pseudo-energy"
$$
\cW[f_\eps](t):=\iint_{\bR^3\times\bR^3}e(\xi)f_\eps(t,x,\xi)dxd\xi+\int_{\bR^3}\tfrac12(|\tilde E_\eps|^2+|\tilde B_\eps|^2)(t,x)dx
$$
is constant under the dynamics of (\ref{VMReg}) --- see Proposition \ref{P-ExistUniqVMeps} --- we have not used this quantity in the 
mean-field limit itself. (The same can be said of earlier work in the same direction, first and foremost \cite{BraunHepp} and \cite{Dobrushin}.)

Strictly speaking, Rein's theorem in \cite{ReinCMS} establishes that $\cW[f_\eps]$ is constant only in the case where the initial particle
distribution function $f^{in}$ is a probability density (in fact, in the more general case where $f^{in}\ge 0$ a.e. is a measurable function
such that $\cW[f^{in}]<+\infty$.) Therefore, even though the empirical measure $f_{N,\eps}$ in Theorem \ref{T-MFLim} is a weak solution 
of the regularized Vlasov-Maxwell system in the sense of measures, Rein's theorem cannot be applied directly to initial data of the form
$f^{in}_N$.

\begin{Prop}\lb{P-ConsEnergN}
Let $N\ge 1$ and $\eps>0$ be fixed, let $(x^{in}_i,\xi^{in}_i)_{1\le i\le N}$ be a $N$-tuple of elements in $\bR^3\times\bR^3$, and let 
$$
f^{in}_N:=\frac1N\sum_{i=1}^N\de_{x^{in}_i}\otimes\de_{\xi^{in}_i}\,.
$$
Let $t\mapsto(x_i(t),\xi_i(t))_{1\le i\le N}$ be the solution of (\ref{DynSystVMReg}) with initial data $(x^{in}_i,\xi^{in}_i)_{1\le i\le N}$, 
and define
$$
f_{N,\eps}(t,\cdot,\cdot):=\frac1N\sum_{i=1}^N\de_{x_i(t)}\otimes\de_{\xi_i(t)}\,,\quad t\ge 0\,.
$$
Then
$$
\cW[f_{N,\eps}](t)=\cW[f^{in}_{N}]
$$
for all $t\ge 0$.
\end{Prop}

\begin{proof}
Regularize $f^{in}_N$ in $(x,\xi)$, for instance by replacing $f^{in}_N$ with its convolution $f^{in}_N\star_{x,\xi}\zeta_\eta$ with any (nonnegative, 
compactly supported) mollifying sequence $(\zeta_\eta)_{\eta>0}$ on $\bR^3\times\bR^3$. If $|x^{in}_i|,|\xi^{in}_i|<R$ for each $i=1,\ldots,N$ and 
$\Supp(\zeta_\eta)\subset B(0,\eta)\times B(0,\eta)$, one has 
$$
\Supp(f^{in}_{N,\eta})\subset B(0,R+\eta)\times B(0,R+\eta)\,.
$$
Then
$$
\tilde E^{in}_{N,\eps}=-\grad_x G\star_x\chi_\eps\star_x\rho_{f^{in}_{N}}
\hbox{ and }
\tilde E^{in}_{N,\eta,\eps}=-\grad_x G\star_x\chi_\eps\star_x\rho_{f^{in}_{N,\eta}}
$$
both belong to $L^2(\bR^3)$ since $\grad_xG(x)=O(|x|^{-2})$ as $|x|\to+\infty$. 

Denote by $(f_{N,\eta,\eps},E_{N,\eta,\eps},B_{N,\eta,\eps})$ the solution of (\ref{VMReg}) with initial data 
$(f^{in}_{N,\eta},\chi_\eps\star_x\tilde E^{in}_{N,\eta,\eps},0)$. Applying Rein's theorem (Proposition \ref{P-ExistUniqVMeps}) shows that
$$
\cW[f_{N,\eta,\eps}](t)=\cW[f^{in}_{N,\eta}]<+\infty\,,\quad t\ge 0\,.
$$

Let $\eta\to 0$. Then $f^{in}_{N,\eta}\to f^{in}_N$ in the weak topology of Borel probability measures on $\bR^3$, and 
$$
\iint_{\bR^3\times\bR^3}e(\xi)f^{in}_{N,\eta}(x,\xi)dxd\xi\to\iint_{\bR^3\times\bR^3}e(\xi)f^{in}_{N}(dxd\xi)
$$
since $\Supp(f^{in}_{N,\eta})\subset B(0,R+\eta)\times B(0,R+\eta)$. On the other hand, denoting
$$
\th_\eta(x):=\int_{\bR^3}\zeta_\eta(x,\xi)d\xi\,,
$$
which is a regularizing sequence on $\bR^3$, we see that
$$
\tilde E^{in}_{N,\eta,\eps}=\th_\eta\star_x\tilde E^{in}_{N,\eps}\to\tilde E^{in}_{N,\eps}
$$
in $L^2(\bR^3)$. Therefore
$$
\cW[f^{in}_{N,\eta}]\to\cW[f^{in}_{N}]<+\infty\quad\hbox{ as }\eta\to 0\,.
$$

Reasoning as in the proof of Theorem \ref{T-MFLim} and Corollary \ref{T-MFLim}, since $f_{N,\eta,\eps}$ and $f_{N,\eps}$ are two solutions 
of the same regularized Vlasov-Maxwell system with initial data satisfying $f^{in}_{N,\eta}\to f^{in}_N$ in weak-$\cP(\bR^3\times\bR^3)$ as 
$\eta\to 0$, we deduce from Dobrushin's estimate that
$$
f_{N,\eta,\eps}(t,\cdot,\cdot)\to f_{N,\eps}(t,\cdot,\cdot)\hbox{ in weak-}\cP(\bR^3\times\bR^3)
$$
uniformly in $t\in [0,T]$ for each $T>0$, while
$$
(\tilde E_{N,\eta,\eps},\tilde B_{N,\eta,\eps})\to(\tilde E_{N,\eps},\tilde B_{N,\eps})
$$
uniformly on $[0,T]\times\bR^3$ for each $T>0$ as $\eta\to 0$. (Recall that $(\tilde E_{N,\eta,\eps},\tilde B_{N,\eta,\eps})$ is the solution
of the regularized Maxwell system (\ref{Maxw}) with initial data (\ref{CondinMaxw}) with $f_{N,\eta,\eps}$ in the place of $f_\eps$, while
$(\tilde E_{N,\eps},\tilde B_{N,\eps})$ is the solution of the same system with $f_{N,\eps}$ in the place of $f_\eps$.)

Obviously $\Supp(f_{N,\eta,\eps}(t,\cdot,\cdot))\subset B(0,R+\eta+t)\times\bR^3$ (since particles travel at speed $<1$). Since
$$
\int_{\bR^3}\rho_{f_{N,\eta,\eps}}(t,x)dx=1\hbox{ and }\int_{\bR^3}|j_{f_{N,\eta,\eps}}(t,x)|dx\le1
$$
we conclude from the formulas expressing $E_{N,\eta,\eps}$ and $B_{N,\eta,\eps}$, i.e.
$$
\ba
E_{N,\eta,\eps}&=-\d_t\grad_xY_\eps\star_xG\star_x\rho_{f^{in}_{N,\eta}}
	-\grad_xY_\eps\star_{t,x}\rho_{f_{N,\eta,\eps}}-\d_tY_\eps\star_{t,x}j_{f_{N,\eta,\eps}}
\\
B_{N,\eta,\eps}&=\Rot_x(Y_\eps\star_{t,x}j_{f_{N,\eta,\eps}})
\ea
$$
that there exists $C_{\eps,T}>0$ such that
$$
|E_{N,\eta,\eps}(t,x)|+|B_{N,\eta,\eps}(t,x)|\le C_{\eps,T}
$$
uniformly in $\eta>0$ as $(t,x)\in[0,T]\times\bR^3$. Therefore, for each $t\in[0,T]$, one has
$$
\Supp(f_{N,\eta,\eps}(t,\cdot,\cdot))\subset B(0,R+\eta+T)\times B(0,R+\eta+2TC_{\eps,T})
$$
uniformly in $\eta>0$, and 
$$
\Supp(\rho_{f_{N,\eta,\eps}}(t,\cdot))\,,\,\,\Supp(j_{f_{N,\eta,\eps}}(t,\cdot))\subset B(0,R+\eta+t)\,.
$$
We also deduce from the formulas expressing $\tilde E_{N,\eta,\eps}$ and $\tilde B_{N,\eta,\eps}$ that, if $0\le t\le T$, then
$$
\tilde E_{N,\eta,\eps}-\d_tY\star_x\tilde E^{in}_{N,\eta,\eps}\hbox{ and }\tilde B_{N,\eta,\eps}
$$
have compact support in $[0,T]\times B(0,R+\eta+\eps+2T)$. 

Since we already know that $E^{in}_{N,\eta,\eps}\to E^{in}_{N,\eps}$ in
$L^2(\bR^3)$ and that $(\tilde E_{N,\eta,\eps},\tilde B_{N,\eta,\eps})\to(\tilde E_{N,\eps},\tilde B_{N,\eps})$ uniformly on $[0,T]\times\bR^3$ 
for each $T>0$ as $\eta\to 0$, we conclude that
$$
\cW[f_{N,\eta,\eps}](t)\to\cW[f_{N,\eps}](t)
$$
for each $t\ge 0$ as $\eta\to 0$, so that
$$
\cW[f_{N,\eps}](t)=\cW[f^{in}_{N,\eps}]
$$
for all $t\ge 0$.
\end{proof}

\smallskip
Since the quantity $\cW[f_{N,\eps}](t)$ is an invariant of the motion for the dynamics of (\ref{DynSystVMReg}), a natural question would be
to check whether (\ref{DynSystVMReg}) is a hamiltonian system. We have left this question unanswered for lack of applications of immediate
interest. Besides, since the mollified dynamics (\ref{DynSystVMReg}) is not a fundamental law of physics, whether it is hamiltonian is perhaps 
a purely academic question.

\smallskip
Several remarks are in order to conclude the analysis presented above.

\smallskip
First, we have chosen to study (\ref{DynSystVMReg}) instead of (\ref{DynSystVMReg2}) because empirical measures constructed on the 
trajectories of  (\ref{DynSystVMReg}) are \textit{exact} solutions of (\ref{VMReg}) in the measure sense. This interpretation disappears if 
one neglects the self-interaction term in the regularized system. However, because of Proposition \ref{P-EquivDynSyst12}, both regularized 
systems are equivalent for $\eps>0$ fixed and $N\to+\infty$. On the contrary, none of these systems seem to make very much sense for
$\eps=0$, and in any case modeling self-interaction in classical electrodynamics is a source of conceptual difficulties, as mentioned above.
As explained on p. 677 in \cite{ElsKiesRic09}, there is little hope of deriving the Vlasov-Maxwell system as a mean-field limit of a particle
system without regularization as long as self-interaction is not better understood.

Also, the work presented here leaves aside several issues mentioned in \cite{ElsKiesRic09} without further analysis, such as including
spin in the particle system; as in \cite{ElsKiesRic09}, the regularization procedure used here (based on convolution with a blob function in 
the space variable only)  is not Lorentz invariant. In any case, this regularization should not be viewed as an attempt to depart from a strict
point particle model and introduce some physical effect involving particle size, but as a mere mathematical expedient.

As a matter of fact, the regularized dynamics (\ref{DynSystVMReg}) could be regarded as an analogue for the Vlasov-Maxwell system of 
the well-known vortex blob method used in numerical simulations of incompressible fluid flows: see section 5.3 in \cite{MarchioPulvi} for 
a concise, yet lucid presentation of this method. In the present context, one has the following approximation result.

\begin{Prop}
Let $f^{in}\in L^\infty(\bR^3\times\bR^3)$ be a compactly supported probability density on $\bR^3\times\bR^3$. 

For each $N\ge 1$, let $(x^{in}_{i,N},\xi^{in}_{i,N})_{1\le i\le N}\in(\bR^3\times\bR^3)^N$ be such that
$$
\frac1N\sum_{i=1}^N\de_{x^{in}_{i,N}}\otimes\de_{\xi^{in}_{i,N}}\to f^{in}dxd\xi\hbox{ weakly in }\cP(\bR^3\times\bR^3)
$$
in the limit as $N\to+\infty$. 

For each $\eps>0$, let $t\mapsto(x(t)_{i,N,\eps},\xi(t)_{i,N,\eps})_{1\le i\le N}$ be the solution of the regularized dynamics (\ref{DynSystVMReg}). 
Then, there exists subsequences $N_k\to+\infty$ and $\eps_k\to 0$ such that
$$
\frac1N\sum_{i=1}^N\de_{x_{i,N_k,\eps_k}(t)}\otimes\de_{\xi_{i,N_k,\eps_k}(t)}\to f(t,\cdot,\cdot)dxd\xi\hbox{ weakly in }\cP(\bR^3\times\bR^3)
$$
uniformly in $t\in[0,T]$ for each $T>0$, where $(f,E,B)$ is a global weak solution of the Vlasov-Maxwell system (\ref{VM}) with initial data
(\ref{CondIn})-(\ref{CompatIn}).
\end{Prop}

The global existence of weak solutions of the Vlasov-Maxwell system was obtained by DiPerna-Lions \cite{DiPernaLionsVM}, with some
additional precisions to be found in \cite{ReinCMS}. The (very weak) approximation result above follows from applying Rein's Proposition 4 in 
\cite{ReinCMS} and Theorem \ref{T-MFLim}. Indeed, let $f_\eps$ be the solution of the regularized Vlasov-Maxwell system (\ref{VMScalReg})
with initial data (\ref{CondInVMScalReg}) for each $\eps>0$. 

By Proposition 4 in \cite{ReinCMS}, there exists a subsequence $\eps_k\to 0$ such that $f_{\eps_k}(t,\cdot,\cdot)dxd\xi\to f(t,\cdot,\cdot)dxd\xi$ 
in weak-$\cP(\bR^3\times\bR^3)$ uniformly in $t\in[0,T]$ for all $T>0$ as $\eps_k\to 0$. Therefore, for each $k>0$ and $T>0$, there exists 
$\eps_k>0$ small enough so that 
$$
\Dist_{MKR}(f(t,\cdot,\cdot),f_{\eps_k}(t,\cdot,\cdot))\le 2^{-k-1}\hbox{ for each }t\in[0,T]\,.
$$
Then, for this value of $\eps_k>0$, by Theorem \ref{T-MFLim}, there exists $N_k>0$ large enough so that 
$$
\Dist_{MKR}(f_{\eps_k}(t,\cdot,\cdot),f_{N_k,\eps_k}(t,\cdot,\cdot))\le 2^{-k-1}\hbox{ for each }t\in[0,T]\,,
$$
where
$$
f_{N_k,\eps_k}(t,\cdot,\cdot):=\frac1N\sum_{i=1}^N\de_{x_{i,N_k,\eps_k}(t)}\otimes\de_{\xi_{i,N_k,\eps_k}(t)}\,.
$$
One then concludes by the triangle inequality.

\smallskip
Finally, it could be interesting to study fluctuations as Braun and Hepp \cite{BraunHepp} did for the Vlasov-Poisson case; this is left for future 
investigation.


\textbf{Acknowledgements.}
The author is indebted to V. Ricci for discussions about this problem and the one analyzed in \cite{ElsKiesRic09}. Part of this work was 
completed during a visit to Kyoto University and it is a pleasure to thank Prof. K. Aoki and the Department of Mechanical Engineering 
and Science for their kind hospitality and support.



\end{document}